\documentclass[11pt]{article}
\usepackage{amsmath,amsfonts,graphics,epsfig}
\usepackage{latexsym}
\usepackage{graphicx}
\usepackage{amssymb}
\usepackage{amsmath}
\usepackage{amsfonts}
\usepackage{color}

\usepackage{times}
\usepackage{amsthm}
\usepackage{bm}

\newcommand{\be}{\begin{equation}}
\newcommand{\ee}{\end{equation}}

\newcommand{\ber}{\begin{eqnarray}}
\newcommand{\eer}{\end{eqnarray}}
\newtheorem{thm}{Theorem}[section]
\newtheorem{prop}{Proposition}[section]
\newtheorem{fact}{F}[section]
\newtheorem{lem}{Lemma}[section]

\newtheorem{rem}{Remark}[section]

\def\unnumfootnote{\xdef\@thefnmark{}\@footnotetext}
\long\def\symbolfootnote[#1]#2{\begingroup\def\thefootnote{\fnsymbol{footnote}}
\footnote[#1]{#2}\endgroup}

\numberwithin{equation}{section}

\def\mean{\hbox{${\rm I \hskip -2pt E}$}}
\def\Pr{\hbox{${\rm I \hskip -2pt P}$}}

\begin{document}
%
\title{On the Convergence of Finite Order Approximations of Stationary Time Series}
%
%
%

\begin{center}
{\Large {\bf On the Convergence of Finite Order Approximations of Stationary Time Series}}
\end{center}
\vspace{1.5cm}

\begin{center}
{ \bf Symantak DATTA GUPTA \footnotemark[1], Ravi R. MAZUMDAR\footnotemark[2] and Peter W. GLYNN\footnotemark[3]
}

\end{center}
\vspace{1cm}
\noindent \footnotemark[1] Department of Electrical and Computer Engineering, University of Waterloo, Waterloo, ON N2L 3G1, Canada
                                                e-mail: sdattagu@engmail.uwaterloo.ca

\noindent \footnotemark[2] Department of Electrical and Computer Engineering, University of Waterloo, Waterloo, ON N2L 3G1, Canada
                                                e-mail: mazum@ece.uwaterloo.ca
                                                
\noindent \footnotemark[3] Department of Management Science and Engineering, Stanford University, Stanford, Ca 94305, USA 
                                                e-mail: glynn@stanford.edu
\vspace{1cm}

\begin{center}
\today
\end{center}
\vspace{0.5cm}

\begin{abstract}

The approximation of a stationary time-series by finite order autoregressive (AR) and moving averages (MA) is a problem that occurs in many applications. In this paper we study asymptotic behavior of the spectral density of finite order approximations of wide sense stationary time series. It is shown that when the on the spectral density is non-vanishing in $[-\pi,\pi]$ and the covariance is summable, the spectral density of the approximating autoregressive sequence converges at the origin. Under additional mild conditions on the coefficients of the Wold decomposition it is also shown that the spectral densities of both moving average and autoregressive approximations converge in $L_2$ as the order of approximation increases.

\end{abstract}
\vspace{0.5cm}

\noindent{\bf Keywords:} Wide sense stationary time series, autoregressive estimate, moving average estimate, spectral density, Wold decomposition, time average variance constant

\vspace{0.3cm}

\vspace{0.5cm}

\noindent{\bf Short-title:} Convergence of spectral density
\vspace{0.5cm}

\section{Introduction}
%
%
%
%
{The} linear estimation of a time series from a finite number of past observations is a problem encountered in many fields of application. This paper deals with the spectral properties of finite order linear approximations for a {\em regular}  zero-mean, real-valued wide sense stationary (WSS) sequence with finite second moment. From the Wold decomposition (\cite{Caines, Shiryayev}), it follows that such a process can be expressed as an infinite weighted sum of unit variance, uncorrelated random variables called the {\em innovation process}. This is a moving average (MA) type model for the process.  By a one-step application of the projection theorem, on the other hand, one obtains an infinite order autoregressive model of the process, wherein the process is expressed as a weighted sum of all its past values, plus the current value of the innovation process.

The problem of approximating an infinite-order process and its spectral density using a finite order model has had a long history of research (\cite{Broersen}) and a detailed survey of the key results in this area is available in \cite{Kay}. A number of problems have been studied related to choosing the optimal order for the approximation for parsimonious modeling as in the work of Akaike through the well known Akaike Information Criterion (AIC, \cite{Akaike}) and the Final Prediction Error criterion (\cite{Akaike69a}, \cite{Akaike70}). Related issues are the estimation of the spectral density through autoregressive (AR) models in the work of Parzen (\cite{Parzen00, Parzen01, Parzen02}) or that of Priestley (\cite{Priestley}) which compares the performance of autoregressive models to that of window based spectral estimators. Among the papers that address the asymptotic behaviour of finite-order estimates as the model order goes to infinity, the results of \cite{Baxter62} are noteworthy. In the works of Pourahmadi (\cite{Pourahmadi89, Pourahmadi93, Pourahmadi}), the nature and rates of convergence of the autoregressive coefficients have been discussed for univariate and multivariate stochastic processes. Similar results on the rate of convergence have been considered in \cite{Findley} and \cite{Poskitt}.

While there are a number of works concerned with the optimal model order; few results are available on the convergence of the spectral density of the approximating finite order autoregressive process.  The main motivation of this paper is to study the asymptotic behavior of the spectral density of finite order approximation models, as the order approaches infinity and to obtain conditions under which the spectral density of the approximation converges to the true spectral density.

The value of the spectral density at the origin plays a special role in the invariance principles for stationary ergodic sequences. Let $\bar{X}_n = \frac{1}{n} \sum_{k=1}^n X_k$ where $\{X_k\}$ is a WSS ergodic process. Let the value of the spectral density be finite at the origin and denote it by $\Gamma^2$.  This quantity is called the Time Average Variance Constant (TAVC) of the process (\cite{Wu}) and according to the central limit theorem due to Ibragimov and Linnik (\cite{Ibragimov}),
\be\nonumber
\sqrt{n} (\bar{X}_n - \mu) \Longrightarrow N(0,\Gamma^2)
\ee
where $\Longrightarrow$ denotes convergence in distribution.

 $\Gamma^2$ plays an important role in steady-state simulations, where the objective is to compute the limit $\lim_{n\to\infty}\bar{X}_n$ when it exists (\cite{Glynn}). One way of estimating $\Gamma^2$ is by windowed estimates based on finite order autoregressive approximations of the observed stationary process. Instead of directly estimating the moments, one obtains approximations for the spectral density at the origin.

In this paper it is shown that when the  spectral density of the process is strictly non-vanishing in $\left(-\frac{1}{2}, \frac{1}{2}\right]$ and its covariance sequence is in $\ell_1$, the spectral density of the approximating autoregressive sequence converges at the origin, as the order of approximation goes to infinity. It is further shown that under the additional condition that the coefficients of the Wold decomposition  of the original process  belong to ${\ell}_1$, spectral densities of both the moving average and the autoregressive estimates converge in $L_2$.

\section{Preliminaries}

Consider a  zero-mean, regular, discrete time, real-valued WSS stochastic process $\lbrace X_n\rbrace_{n\in\mathbb{Z}}$ defined on a probability space ($\Omega$, $\mathfrak{F}$, $\Pr$). Let $L_2(\Pr)$ denote the Hilbert  space of random variables with finite second moment with the inner product $\mean[\cdot,\cdot]$ defined thereon. Two zero mean\footnote{In this paper it is always assumed that the underlying random variables are of zero mean} random variables $X, Y\in L_2(\Pr)$ are said to be orthogonal if $\mean[XY] = 0$. Let $R_k = \mean[X_{n+k}X_n]$  denote the covariance sequence.

The spectral density $S(\lambda)$ is defined as the Fourier transform of the covariance sequence. 
 \be\nonumber
 S(\lambda) = \sum_{k\in\mathbb{Z}} R_k e^{-2\pi i\lambda k},\mbox{  }\mbox{  }\mbox{  }\lambda \in \left(-\frac{1}{2},\frac{1}{2}\right] \\ \ee
 The covariance sequence can be recovered from the spectral density using the inverse Fourier Transform.
 \be \nonumber
 R_k = \int_{-\frac{1}{2}}^{\frac{1}{2}} S(\lambda) e^{2\pi i\lambda k}d\lambda,\ \ \mbox{  }\mbox{  } k\in\mathbb{Z}
 \ee

 For $n\in\mathbb{Z}$ define the space $H_n =$ linear span of $\lbrace X_n, X_{n-1}, X_{n-2}, \ldots \rbrace$, i.e., the Hilbert subspace generated by the closure of the linear combinations of $X_n$ and all its past values at time $n$.

Let $Y$ be a random variable defined on $L_2(\Pr)$.  Define $\overline{\mean}[Y|H_n]$ as the projection of $Y$ onto the space $H_n$ with respect to the inner-product $\mean[\cdot,\cdot]$. Then $\overline{\mean}[Y|H_n]$ is the minimum mean squared error (MMSE) linear estimate of $Y$ given $H_n$.

 Define $\lbrace \nu_n\rbrace_{n\in\mathbb{Z}}$ as the innovation process associated with $\lbrace X_n\rbrace_{n\in\mathbb{Z}}$, i.e.,
  \begin{equation}\nonumber
 \nu_n = X_n - \overline{\mean}[X_n|H_{n-1}]
 \end{equation}
  with $\mean[\nu_n]=0$ and $\mean[\nu_n\nu_k]  = \sigma^2 \delta_{n-k}$, where $\delta_k$ denotes the Kronecker delta. Without loss of generality it is assumed that $\sigma^2=1$. By construction, $\nu_n$ is orthogonal to $H_{k-1}$,  for all $k\le n$. The sequence $\lbrace \nu_k \rbrace _{k\in\mathbb{Z},k\le n}$ spans the subspace $H_n$ and constitutes an orthogonal basis in $L_2(\Pr)$  for the latter. Note that $\overline{\mean}[X_n|H_{n-1}]$ corresponds to the MMSE linear estimate of $X_n$ given the space $H_{n-1}$ and it can therefore be written as a linear combination of all the past values of $X_{n}$ as follows:
\begin{equation}\label{prel-3}
\overline{\mean}[X_n|H_{n-1}] = \sum_{k=1}^{\infty}b_kX_{n-k}
\end{equation}
where the $b_k$s minimize the  mean squared error. The corresponding mean squared error is
\begin{eqnarray}\nonumber
\mean[(X_n-\sum_{k=1}^{\infty}b_k X_{n-k})^2]&=& \mean[(X_n-\overline{\mean}[X_n|H_{n-1}])^2]\notag\\
&=& \mean[\nu_n^2] = 1 \mbox{ for all }n \end{eqnarray}
 $\overline{\mean}[X_n|H_{n-1}]$ is thus an infinite-order autoregressive estimate of $X_n$. On the other hand, being a wide sense stationary process, $X_n$ may be expressed in terms of its Wold decomposition as follows (\cite{Caines, Shiryayev}).
\begin{equation}\label{definewold}
X_n = \sum_{k=0}^\infty a_k\nu_{n-k}\mbox{ for all }n\in\mathbb{Z}
\end{equation}
\begin{equation}\nonumber
\mbox{ with} \sum_{k=0}^\infty |a_k|^2 < \infty  \mbox{  and } a_0 =1.
\end{equation}
This gives an infinite order moving average representation of $X_n$. The coefficients $\{b_k\}$ and $\{a_k\}$ are related to each other as follows. For each $k$,
\be \label{akbk} b_k = \sum_{j=1}^k a_jb_{k-j}\ee

From (\ref{definewold}) and the properties of the innovations sequence it is seen that:
\be\nonumber
R_k = \sum_{n=0}^{\infty} a_na_{n-k}
\ee
Therefore, if the spectral density is finite at the origin then:
\be
\nonumber
\sum_{k\in\mathbb{Z}} R_k = \sum_{k\in\mathbb{Z}}\sum_{n=0}^{\infty} a_na_{n-k}  = \left(\sum_{k=0}^\infty a_k\right)^2 < \infty
\ee

Let $p$ be a positive integer, and define the space $H_n^p =$ linear span of $\{ X_n$, $X_{n-1}$, $X_{n-2}$, $\ldots$ $X_{n-p+1}\}$, i.e., the space of all linear combinations of the $p$ most recent values of the sequence at time $n$, including $X_n$; and all their limits when they exist. $H_n^p$ is then a closed Hilbert subspace of $H_n$, for all $p$ and $n$. Clearly, $H_n^\infty = H_n$, by definition.

Define $\overline{X}_{n,p}$ as the MMSE autoregressive approximation of $X_n$ of order $p$. Then $\overline{X}_{n,p}$ is a linear combination of $\lbrace X_{n-1}, X_{n-2}, \ldots, X_{n-p}\rbrace$ and is given by

\be \label{res-2}
\overline{X}_{n,p}= \overline{\mean}[X_n|H_{n-1}^p] = \sum_{k=1}^{p}b_{k,p}X_{n-k}
\ee
where the coefficients $b_{k,p}$ minimize the error $\mean[(X_n-\sum_{k=1}^{p}b_{k,p}X_{n-k})^2]$. The coefficients $b_{k,p}$ can be obtained as solutions to the modified Yule-Walker equations (\cite{Picinbono, Kailath}) given by:
\be \nonumber\bm{B_p}=\bm{R_p}^{-1}\bm{r_p} \ee
where \be\nonumber
\bm{R_p} =
 \begin{pmatrix}
  R_0 & R_1 & \cdots & R_{p-1} \\
  R_1 & R_0 & \cdots & R_{p-2} \\
  \vdots  & \vdots  & \ddots & \vdots  \\
  R_{p-1} & R_{p-2} & \cdots & R_0
 \end{pmatrix}
\ee
 \be \nonumber \bm{r_p} = [R_1 \mbox{ }R_2 \mbox{ } \ldots \mbox{ } R_p]\ee
 and \be \nonumber \bm{B_p} = [b_{1,p}\mbox{ }\ldots\mbox{ }b_{p,p}]\ee

\section{Moving average approximations of regular stationary sequences}

Consider a moving average approximation of $X_n$ of order $p$, constructed using the innovation sequence $\{\nu_n\}$.
Note by assumption, $\mbox{Var}[\nu_n]=\sigma_\nu^2=1$ for all $n\in\mathbb{N}$. We first state some known facts without proof.

 \begin{fact}
 The best $p$th order moving average approximation of $X_n$ is given by $\hat{X}_{n,p} = \sum_{k=0}^p a_k\nu_{n-k}$.
\end{fact}
The next result is related to the mean square convergence of the p-th order approximation that readily follows from the properties of the innovation process and the Wold decomposition.

\begin{fact}As $p\to\infty$, $\hat{X}_{n,p}$ converges to $X_n$ in quadratic mean.\end{fact}

\subsection{Convergence of the Spectral Density in $L_2$}
For all functions $F: \left(-\frac{1}{2}, \frac{1}{2}\right]\to\mathbb{C}$ such that $\int_{-\frac{1}{2}}^{\frac{1}{2}}|F(\lambda)|^2d\lambda < \infty$, define $||\cdot||$ to be the $L_2$ norm as follows:
\begin{equation}\nonumber
||F(\lambda)|| = \left|\int_{-\frac{1}{2}}^{\frac{1}{2}}|F(\lambda)|^2d\lambda\right|^{\frac{1}{2}}
\end{equation}

\begin{prop}
Under the condition  $\sum_{k=0}^\infty |a_k|< \infty$ where $\{a_k\}$ correspond to the coefficients in the Wold decomposition, 
the spectral density of $\hat{X}_{n,p}$ converges uniformly in $L_2$ to that of $X_n$ as $p\to\infty$.
\end{prop}

\begin{proof}

Let $\sum_{k=0}^\infty |a_k| = S$.
It follows from the properties of the innovation process, that the covariance sequence of $\{X_n\}$ is $\{R_k\}$ where
\be\nonumber R_k = \sum_{l=0}^\infty a_la_{l-k}\ee

The spectral density $S_X(\lambda)$ is given by
\ber
S_X(\lambda)&=& \sum_{k\in\mathbb{Z}} R_ke^{-2\pi ik\lambda}
\notag\\
 &=& \sum_{k\in\mathbb{Z}}\sum_{l=0}^\infty a_la_{l-k}e^{-2\pi ik\lambda}\notag\\
 &=& A_0^\infty(\lambda)A_0^\infty(-\lambda)\notag\eer
where for $M, N \in \mathbb{N}$  \be\nonumber A_M^N(\lambda)=\sum_{l=M}^N a_le^{-2\pi il\lambda}\nonumber\ee and
\be\nonumber A_M^\infty(\lambda)=\sum_{l=M}^\infty a_le^{-2\pi il\lambda}\nonumber\ee
Similarly, the spectral density of $\hat{X}_{n,p}$ is given by
\begin{equation}\nonumber
S_{\hat{X}_p}(\lambda) = A_0^p(\lambda)A_0^p(-\lambda)
\end{equation}
Now consider
\be\label{norm}\begin{split}
&\!
\begin{aligned}[t]
&& ||S_X(\lambda)-S_{\hat{X}_p}(\lambda)||\end{aligned}\\
&\!
\begin{aligned}[t]
  =  ||A_0^\infty(\lambda)A_0^\infty(-\lambda)-A_0^p(\lambda)A_0^p(-\lambda)||\end{aligned}\\
&\!
\begin{aligned}[t]
  =  ||(A_0^p(\lambda)A_{p+1}^\infty(-\lambda)+A_{p+1}^\infty(\lambda)A_0^\infty(-\lambda)||\end{aligned}\\
&\!
\begin{aligned}[t]
  \le ||(A_0^p(\lambda)A_{p+1}^\infty(-\lambda)||+||A_{p+1}^\infty(\lambda)A_0^\infty(-\lambda)||
 \end{aligned}\\
&\!
\begin{aligned}[t]
 =\left|\int_{-\frac{1}{2}}^{\frac{1}{2}}I_1(p,\lambda) d\lambda\right|^{\frac{1}{2}}+\left|\int_{-\frac{1}{2}}^{\frac{1}{2}}I_2(p,\lambda) d\lambda\right|^{\frac{1}{2}}
 \end{aligned}\\
\end{split}
\ee
where
\begin{equation}\nonumber
I_1(p,\lambda)= |A_0^p(\lambda)A_{p+1}^\infty(-\lambda)|^2
\end{equation}
\begin{equation}\nonumber
I_2(p,\lambda)= |A_{p+1}^\infty(\lambda)A_0^\infty(-\lambda)|^2
\end{equation}
Then,
\ber I_1(p,\lambda) &=& |A_0^p(\lambda)|^2|A_{p+1}^\infty(-\lambda)|^2\notag\\
&\le& \left(\sum_{k=0}^p |a_k|\right)^2\left(\sum_{k=p+1}^\infty |a_k|\right)^2
 \notag\\
 &\le& S^2\left(\sum_{k=p+1}^\infty |a_k|\right)^2\notag\eer
and similarly, 
\be\nonumber I_2(p,\lambda)\le S^2\left(\sum_{k=p+1}^\infty |a_k|\right)^2 \ee
Since the sum $\sum_{k=0}^\infty |a_k|$ is non-decreasing and converges to $S<\infty$, for any given $\epsilon>0$, there exists a positive integer $p$ such that $\sum_{k=p+1}^\infty |a_k|<\frac{\epsilon}{2S}$, so that
\begin{equation}\label{I1}
I_1(p,\lambda)<\frac{\epsilon^2}{4}
\end{equation}
\begin{equation}\label{I2}
I_2(p,\lambda)<\frac{\epsilon^2}{4}
\end{equation}

Combining (\ref{I1}) and (\ref{I2}) in (\ref{norm}) yields
\be\nonumber
||S_X(\lambda)-S_{\hat{X}_p}(\lambda)|| < \left(\left|\int_{-\frac{1}{2}}^{\frac{1}{2}}\frac{\epsilon^2}{4} d\lambda\right|^{\frac{1}{2}}+\left|\int_{-\frac{1}{2}}^{\frac{1}{2}}\frac{\epsilon^2}{4} d\lambda\right|^{\frac{1}{2}}\right)
 = \epsilon
 \ee

Thus, for any $\epsilon>0$, there exists a positive integer $p$ such that $||S_X(\lambda)-S_{\hat{X}_p}(\lambda)||<\epsilon$ for all $\lambda$. Therefore, $S_{\hat{X}_p}(\lambda)$ converges uniformly in $L_2$ to $S_X(\lambda)$ as $p\to \infty$.
\end{proof}

\section{Finite AR approximations of regular stationary sequences }
We begin by looking at the asymptotic behaviour of the autoregressive approximation given by (\ref{res-2}).  Define $\nu_{n,p}$ as the error in estimation corresponding to the AR-$p$ approximation of $\{X_n\}$, i.e.,

\be \nonumber \nu_{n,p} = X_n - \mean[X_n|H_{n-1}^p]\ee

Note that for all $p\in\mathbb{N}$
\begin{eqnarray}\label{enun}
\overline{\mean}[\nu_{n,p}|H_{n-1}]&=& \overline{\mean}[(X_n - \overline{\mean}[X_n|H_{n-1}^p])|H_{n-1}]\notag\\
&=& \overline{\mean}[X_n|H_{n-1}]-\overline{\mean}[X_n|H_{n-1}^p] \notag\\
&=&(X_n-\overline{\mean}[X_n|H_{n-1}^p])\notag\\
&-& (X_n-\overline{\mean}[X_n|H_{n-1}])\notag\\
 &=& \nu_{n,p} - \nu_n\end{eqnarray}

Let $\overline{\mean}[\nu_{n,p}|H_{n-1}] = \epsilon_{n,p}$. Then
\begin{eqnarray}
\epsilon_{n,p} &=& \nu_{n,p} - \nu_n \mbox{ and}\notag\\
 \mean[\epsilon_{n,p}] &=& \mean[\nu_{n,p} - \nu_n] = 0\notag \end{eqnarray}
As $X_n$ is a WSS sequence in $n$ and $\nu_{n,p}$ is constructed as a linear combination of $X_n$, $X_{n-1}$, $\ldots$, $X_{n-p}$ whose coefficients do not depend on $n$,  it follows that $\nu_{n,p}$ is also a WSS sequence in $n$. The variance of $\nu_{n,p}$, then, is only a function of p. Let $\mbox{Var}[\nu_{n,p}] = \sigma_p^2$. By (\ref{enun}), and the fact that $\nu_n$ is orthogonal to the subspace $H_{n-1}$ (and hence to $\epsilon_{n,p}$) we obtain:
\begin{eqnarray}\sigma_p^2 &=& \mbox{Var}[\nu_n] + \mbox{Var}[\epsilon_{n,p}]\notag\\
&=& 1 + \mbox{Var}[\epsilon_{n,p}]\notag\\
&=& 1 + \mean[\epsilon_{n,p}^2]\notag\\
&\ge& 1 \mbox{ for all $p\in\mathbb{N}$}
 \notag\end{eqnarray}
Note that for any  $q$, $p\in\mathbb{N}$ such that $q > p$,  $H_{n-1}^p\subset H_{n-1}^q$. It follows then, that $\overline{\mean}[X_n|H_{n-1}^q]$ is at least as good a linear estimate of $X_n$ as $\overline{\mean}[X_n|H_{n-1}^p]$, in terms of mean squared error. Therefore,
\be\nonumber
\mean[(X_n -\overline{\mean}[X_n|H_{n-1}^q])^2] \le \mean[(X_n -\overline{\mean}[X_n|H_{n-1}^p])^2]
\ee
and hence
\be \nonumber\sigma_q^2 \le \sigma_p^2  \ee
Hence, the sequence  $\sigma_p^2$ is a non-increasing sequence in $p$, bounded below by 1 and must therefore has a limit as $p \to \infty$ that is bounded from below by 1. It can be shown that the limit  is in fact equal to 1. 

\begin{rem}\label{P1}
 As $p \to \infty$, $\overline{\mean}[X_n|H_{n-1}^p]  \to \overline{\mean}[X_n|H_{n-1}]$ and  $\nu_{n,p} \to \nu_n$ in quadratic mean. \end{rem}
 
 The proof is a direct application of  \cite[Lemma 3.1(b)]{Caines}.

\subsection{Convergence of the Spectral Density at the Origin}

We now study conditions under which the spectral density of the finite order AR approximation converges at the origin. As mentioned $S(0)$ is referred to as the Time Average Variance Constant (TAVC) and plays an important role in simulations. Throughout the following it is assumed that the spectral density $S(\lambda)$ is non-vanishing in $\left(-\frac{1}{2}, \frac{1}{2}\right]$ i.e. there exists $\varepsilon  > 0$ such that $S(\lambda) > \varepsilon$ for all $\lambda \in \left(-\frac{1}{2}, \frac{1}{2}\right]$.

We start with the following lemma on the pointwise convergence of the coefficients $b_{k,p}$ (\cite[Proposition 3.1]{Degerine}).

\begin{lem}\label{pw}As $p\to\infty$, $b_{k,p}\to b_k$ for each $k\in\mathbb{N}$. 
\end{lem}

\begin{proof}
For each $p\in\mathbb{N}$ and $k\in\{1,\dots,p\}$ define $\tilde{b}_{0,p} = 1$ and $\tilde{b}_{k,p} = - b_{k,p}$. For any $p$, $\nu_n^p$ is given by
\be \nonumber \nu_{n,p} = \sum_{k=0}^p\tilde{b}_{k,p}X_{n-k}\ee
The above can be written in matrix form as $\mathbf{\nu} = B\mathbf{X}$ where $\mathbf{\nu} = [\nu_{n,0} \cdots \nu_{n,p}]^T$, $\mathbf{X} = [X_n \cdots X_{n-p}]^T$ and $B$ is a lower triangular matrix whose first column is $[\tilde{b}_{0,p} \cdots \tilde{b}_{p,p}]$. The matrix $B$ is invertible with inverse $A$ which satisfies $\mathbf{X} = A\mathbf{\nu}$. The inverse $A$ is lower triangular with first column $[a_{0,p} \cdots a_{p,p}]$ and the elements $a_{i,j}$ satisfy for all $n\in\mathbb{N}$
\be \label{aX} X_n = \sum_{k=0}^p a_{k,p}\nu_{n-k, p-k}\ee
and for $p>k$
\be \label{ab}  \tilde{b}_{k,p} = -\sum_{j=1}^k a_{j,p}\tilde{b}_{k-j, p-j}\ee
By definition, $\nu_{n,n}$ and $\nu_{m,m}$ are orthogonal to each other for $n\ne m$ and hence (\ref{aX}) provides an orthogonal decomposition of $X_n$. Therefore, 
\be \label{sigmalim} \mean[X_n\nu_{n-k, p-k}] = a_{k,p}\sigma^2_{p-k}\ee 
From (\ref{definewold}), on the other hand, we have
\be \label{sigmanu} \mean[X_n\nu_{n-k}] = a_{k}\ee 
However, by remark \ref{P1}, $\nu_{n,p} \to \nu_n$ in quadratic mean. It then follows from (\ref{sigmalim}) and (\ref{sigmanu}) that for all $k\in \mathbb{N}$
\be \begin{split}
&\!
\begin{aligned}[t] && &\lim_{p\to\infty}& |a_{k,p}\sigma^2_{p-k}- a_k|\end{aligned}\\
&\!
\begin{aligned}[t] = &\lim_{p\to\infty}& |\mean[X_n\nu_{n-k, p-k}-X_n\nu_{n-k}]|\end{aligned}\\
&\!
\begin{aligned}[t]\le &\lim_{p\to\infty}& \left|\mean[X_n^2]\mean[(\nu_{n-k, p-k}-\nu_{n-k})^2]\right|^{\frac{1}{2}} = 0\end{aligned}\\
\end{split}
\ee
and hence,  
\be \label{lima} \lim_{p\to\infty} a_{k,p} = a_k\ee
Finally, to show the pointwise convergence of $\tilde{b}_{k,p}$ (and therefore that of $b_{k,p}$) first observe that $\tilde{b}_{0,p}= b_0 =1$ holds for all $p$. Let \be\nonumber \lim_{p\to\infty} \tilde{b}_{j,p}= -b_j\ee for all $j\le k$. Then, using the recursive relation given by (\ref{ab}) and comparing with (\ref{akbk}), one obtains \be\nonumber \lim_{p\to\infty} \tilde{b}_{k+1,p}= -b_{k+1}\ee Therefore, by the principle of mathematical induction,  as $p\to\infty$, $\tilde{b}_{k,p} \to -b_k$; i.e., $ b_{k,p} \to b_k$ for each $k \in \mathbb{N}$. 
\end{proof}
Next, we present a key result on the summability of the autoregressive coefficients known as Baxter's inequality \cite[Theorem 2.2]{Baxter62}).

Let $\{X_n\}$ be a WSS process with spectral density function $S_X(\lambda) >0$ and let $\overline{X}_{n,p}$ be the $p$-th order MMSE linear predictor of $X_n$, defined by \ref{res-2} and let $\sigma_p^2$ be the corresponding mean squared error. Let $\{b_k\}$ be the limits of the coefficients $\{b_{k,p}\}$ and let $\sigma^2 > 0$ be the limit of $\sigma^2_p$ as $p\to\infty$ (in our case $\sigma^2 =1$ by remark \ref{P1}). Define the sequence $\{u_{k,p}\}$ as $u_{k,p} = -\frac{b_{k,p}}{\sigma^2_p}$ and let $\{U_k\}$ be the limit of $\{u_{k,p}\}$. Then, the theorem is stated as follows.

\begin{thm}\label{Baxterin}\textbf{Baxter's Inequality}: If $S_X(\lambda)$ is a positive continuous function whose Fourier coefficients have $\gamma$ moments, then there exists an integer $N>0$ and a constant $c>0$, both depending only on $S_X(\lambda)$ such that for all $p\ge N$,

\be \nonumber \sum_{k=1}^p(2^\gamma + k^\gamma)|u_{k,p}-U_k| \le c\sum_{k=p+1}^\infty(2^\gamma + k^\gamma)|U_k|\ee
\end{thm}

Note that the Fourier coefficients of the spectral density are the elements of the covariance sequence $\{R_k\}$.

The above theorem can be used to establish the following lemma on the convergence of the coefficients $\{b_{k,p}\}$ as $p\to\infty$.

\begin{lem}\label{Baxter}
When the spectral density of $\{X_n\}$ is strictly positive in $\lambda\in \left(-\frac{1}{2}, \frac{1}{2}\right]$, and the covariance sequence is in $\ell_1$, i.e.,
\be\nonumber \sum_{k\in \mathbb{Z}} |R_k| <\infty\ee
then \be\nonumber \lim_{p\to \infty}\sum_{k=1}^p |b_{k,p}-b_k| = 0 \ee
\end{lem}

\begin{proof}
The proof is a simple application of Baxter's inequality. Note that when the covariance sequence is in $\ell_1$, the spectral density is continuous. Pointwise convergence of the $b_{k,p}$ to $b_k$ for each $k$ follows from lemma \ref{pw}. Moreover, summability of the covariance sequence also imply that the sequence has a finite $0$-th moment ($\gamma = 0$ in Theorem \ref{Baxterin}). It then follows from Theorem \ref{Baxterin} that under the assumption that the spectral density of $\{X_n\}$ is strictly positive in $\lambda\in \left(-\frac{1}{2}, \frac{1}{2}\right]$, there exists a positive integer $N$ and a constant $c>0$ such that
\be \sum_{k=1}^p\left|\frac{b_{k,p}}{\sigma_p^2}-b_k\right| \le c\sum_{k=p+1}^\infty|b_k| \ee
for all $p>N$. Since the covariance sequence has been assumed to be in $\ell_1$, the sequence of autoregressive coefficients of the original process are also in $\ell_1$ (\cite{Bhansali}), i.e.,

\be \nonumber\sum_{k=1}^\infty |b_k| <\infty\ee
Then, for any $\epsilon>0$, there exists an integer $N_0$ such that
\be \nonumber\sum_{k=p+1}^\infty |b_k| <\frac{\epsilon}{c} \ee
for all $p>N_0$.
Define $N^* = \max \{N, N_0\}$. Then for all $p>N^*$,
\be \nonumber\sum_{k=1}^p\left|\frac{b_{k,p}}{\sigma_p^2}-b_k\right| < \epsilon \ee

Therefore,
\be\label{Baxter1} \lim_{p\to\infty}\sum_{k=1}^p\left|\frac{b_{k,p}}{\sigma_p^2}-b_k\right| = 0\ee

It follows from the triangle inequality that
\ber \sum_{k=1}^p\left|\frac{b_{k,p}-b_k}{\sigma_p^2}\right|&\le& \sum_{k=1}^p\left|\frac{b_{k,p}}{\sigma_p^2}-b_k\right| + \sum_{k=1}^p\left|\frac{b_{k}}{\sigma_p^2}-b_k\right| \notag\\
&\le& \sum_{k=1}^p\left|\frac{b_{k,p}}{\sigma_p^2}-b_k\right| + \frac{\sigma_p^2-1}{\sigma_p^2}\sum_{k=1}^p|b_k|\notag\eer
Therefore, as $p\to\infty$,
\ber \lim_{p\to \infty} \sum_{k=1}^p\left|\frac{b_{k,p}-b_k}{\sigma_p^2}\right|&\le& \lim_{p\to \infty}\sum_{k=1}^p\left|\frac{b_{k,p}}{\sigma_p^2}-b_k\right| \notag\\ &+& \lim_{p\to \infty}\frac{\sigma^2_p-1}{\sigma_p^2}\sum_{k=1}^p|b_k| \notag\eer
By (\ref{Baxter1}), the first limit on the right hand side is zero and by remark \ref{P1}, 
\be \nonumber \lim_{p\to \infty} \sigma_p^2 = 1\ee  
Therefore,
\be \nonumber\lim_{p\to \infty} \sum_{k=1}^p|b_{k,p}-b_k|=0 \ee

\end{proof}

\begin{prop}\label{lemmapsd}
Let $S_X(\lambda)>0$ for $\lambda\in\left(-\frac{1}{2}, \frac{1}{2}\right]$ and let $\sum_{k\in\mathbb{Z}}|R_k| < \infty$.  Then, as $p \to \infty$, the spectral density of $\overline{\mean}[X_n|H_{n-1}^p] = \overline{X}_{n,p}$ converges to that of $\overline{\mean}[X_n|H_{n-1}]$ at the origin.
\end{prop}
\begin{proof}
  Let $\lbrace\overline{R}_{k,p}\rbrace_{k\in\mathbb{Z}}$, $\lbrace \overline{R}_{k}\rbrace_{k\in\mathbb{Z}}$ be the covariance sequences of $\overline{\mean}[X_n|H_{n-1}^p]$ , $\overline{\mean}[X_n|H_{n-1}]$  respectively. If $S_{\overline{X}_p}(\lambda)$, $S_{\overline{X}}(\lambda)$ be the spectral densities of the two processes respectively, then
\begin{equation}\nonumber
S_{\overline{X}_p}(\lambda) = \sum_{k\in\mathbb{Z}}\overline{R}_{k,p}e^{-2\pi i \lambda k}
\end{equation}
\begin{equation}\nonumber
S_{\overline{X}}(\lambda) = \sum_{k\in\mathbb{Z}}\overline{R}_{k}e^{-2\pi i\lambda k}
\end{equation}
At the origin, i.e., at $\lambda=0$, the above spectral densities are given by
\begin{equation}\nonumber
S_{\overline{X}_p}(0)= \sum_{k\in\mathbb{Z}}\overline{R}_{k,p}
\end{equation}
\begin{equation}\nonumber
S_{\overline{X}}(0) = \sum_{k\in\mathbb{Z}}\overline{R}_{k}
\end{equation}
For some $p$ and $k$, $\overline{R}_{k,p}$ may be obtained from (\ref{res-2}) as

\begin{eqnarray}\label{final0}
\overline{R}_{k,p} &=& \mean[\overline{X}_{n,p}\overline{X}_{n-k,p}]
\notag\\
&=&\mean\left[\sum_{j=1}^{p}b_{j,p}X_{n-j}\sum_{l=1}^{p}b_{l,p}X_{n-k-l}\right]\notag\\
 &=&\sum_{j=1}^{p}b_{j,p}^2R_k+\sum_{t=1}^{p-1}\sum_{j=1}^{p-t}b_{j,p}b_{j+t,p}(R_{k-t}+R_{k+t})\notag\\ \end{eqnarray}
Summing the above over all $k\in\mathbb{Z}$ gives
\begin{eqnarray}\sum_{k\in\mathbb{Z}}\overline{R}_{k,p} &=& \left(\sum_{j=1}^{p}b_{j,p}^2 + 2\sum_{t=1}^{p-1}\sum_{j=1}^{p-t}b_{j,p}b_{j+t,p}\right)\sum_{k\in\mathbb{Z}}R_k
\notag\\
&=&\left(\sum_{j=1}^{p}b_{j,p}\right)^2\sum_{k\in\mathbb{Z}}R_k\notag\\
\end{eqnarray}

 From where it follows that
\begin{equation}\label{limit bkp}
\lim_{p\to\infty}\sum_{k\in\mathbb{Z}}\overline{R}_{k,p} = \lim_{p\to\infty}\left(\sum_{j=1}^{p}b_{j,p}\right)^2\sum_{k\in\mathbb{Z}}R_k\end{equation}

Proceeding similarly, using the expression in (\ref{prel-3}), one obtains the following expression for $\sum_{k\in\mathbb{Z}}\overline{R}_{k}$:

\be \label{sum Rk} \sum_{k\in\mathbb{Z}}\overline{R}_{k} =  \left(\sum_{j=1}^{\infty}b_j\right)^2\sum_{k\in\mathbb{Z}}R_k\ee

Recall that by lemma \ref{Baxter}, under the stated conditions,
\be \lim_{p\to\infty}\sum_{k=1}^p\left|b_{k,p}-b_k\right| = 0\ee

Clearly, then,
\be \lim_{p\to\infty}\sum_{k=1}^p (b_{k,p}-b_k) = 0\ee
i.e.,
\be\label{sum bkp} \lim_{p\to\infty}\left(\sum_{k=1}^p b_{k,p}\right)^2  = \left(\sum_{k=1}^\infty b_k\right)^2\ee

Combining (\ref{sum bkp}) with (\ref{limit bkp}), and comparing with (\ref{sum Rk}), we obtain
\ber \nonumber\lim_{p\to\infty}\sum_{k\in\mathbb{Z}}\overline{R}_{k,p} &=&  \left(\sum_{k=1}^\infty b_k\right)^2\sum_{k\in\mathbb{Z}}R_k\notag\\
&=&  \sum_{k\in\mathbb{Z}}\overline{R}_k \eer

This completes the proof.

\end{proof}
\vspace{0.4cm}

\begin{rem}
Combining the above result with the fact that $\nu_{n,p}$ converges to $\nu_n$ in quadratic mean it readily follows that the spectral density of $\{\overline{X}_{n,p}+\nu_{n,p}\}$ converges to the spectral density of $\{X_n\}$ at the origin, i.e., at $\lambda=0$.
\end{rem}

\subsection{Convergence of the Spectral Density in $L_2$}

We now present a sufficient condition for the $L_2$ convergence of the spectral density of the autoregressive approximation as $p\to\infty$.

\begin{prop}\label{infautoL2}
 Let $S_X(\lambda)>0$ for $\lambda\in\left(-\frac{1}{2}, \frac{1}{2}\right]$ and let \be\sum_{k\in\mathbb{Z}}|a_k|<\infty\ee Then as $p \to \infty$, $S_{\overline{X}_p}(\lambda)$ converges to $S_{\overline{X}}(\lambda)$ in $L_2$.
\end{prop}

\begin{proof}
We begin by noting that when the sequence $\{a_k\}$ is in $\ell_1$, (as stated above), both $\{R_k\}$ and $\{b_k\}$ are also in $\ell_1$ and therefore the conditions of Lemma \ref{Baxter} are satisfied.

Refer to equation $(\ref{final0})$ for an expansion of $\overline{R}_{k,p}$ for each $k$, for a given $p$:
\be\nonumber
\overline{R}_{k,p} =\sum_{j=1}^{p}b_{j,p}^2R_k+\sum_{t=1}^{p-1}\sum_{j=1}^{p-t}b_{j,p}b_{j+t,p}(R_{k-t}+R_{k+t})\\
  \ee

Consider the WSS process given by $\sum_{j=1}^p b_k X_{n-k}$
and let $\{R_{k,p}\}$ be its covariance sequence. Proceeding as in the case of (\ref{final0}), we can obtain a similar expression for $R_{k,p}$ as follows.
  \be
R_{k,p} =\sum_{j=1}^{p}b_{j}^2R_k+\sum_{t=1}^{p-1}\sum_{j=1}^{p-t}b_{j}b_{j+t}(R_{k-t}+R_{k+t})\\
  \ee
  so that for all $k\in\mathbb{Z}$
 \begin{equation}\begin{split}
&\!
\begin{aligned}[t] && |R_{k,p} - \overline{R}_{k,p}| \end{aligned}\\
&\!
\begin{aligned}[t]= \bigg|\sum_{j=1}^{p}(b_{j}^2 -b_{j,p}^2) R_k \end{aligned}\\
&\!
\begin{aligned}[t]+ \sum_{t=1}^{p-1}\sum_{j=1}^{p-t}(b_{j}b_{j+t}-b_{j,p}b_{j+t,p})\left(R_{k-t}+R_{k+t}\right)\bigg| \end{aligned}\\
&\!
\begin{aligned}[t]\le  \sum_{j=1}^{p}\left| (b_{j}^2 -b_{j,p}^2)\right| |R_k|\end{aligned}\\
&\!
\begin{aligned}[t]+ \sum_{t=1}^{p-1}\sum_{j=1}^{p-t}\left|(b_{j}b_{j+t}-b_{j,p}b_{j+t,p})\right|(|R_{k-t}|+|R_{k+t}|)\end{aligned}
\end{split}
\end{equation}
 Summing over all $k\in\mathbb{Z}$
   \begin{equation}\begin{split}
&\!
\begin{aligned}[t] && \sum_{k\in\mathbb{Z}}|R_{k,p} - \overline{R}_{k,p}|\end{aligned}\\
&\!
\begin{aligned}[t]
  \le
\bigg( \sum_{j=1}^{p}\left|b_{j}^2 -b_{j,p}^2\right| \end{aligned}\\
&\!
\begin{aligned}[t] + 2\sum_{t=1}^{p-1}\sum_{j=1}^{p-t}\left|b_{j}b_{j+t}-b_{j,p}b_{j+t,p}\right| \bigg)\sum_{k\in\mathbb{Z}}|R_k|\end{aligned}\\
&\!
\begin{aligned}[t]= \left( \sum_{j=1}^{p}\sum_{i=1}^{p}\left|b_{i}b_{j} -b_{i,p}b_{j,p}\right|\right)\sum_{k\in\mathbb{Z}}|R_k|\end{aligned}\\
&\!
\begin{aligned}[t]\le \bigg( \sum_{j=1}^{p}\sum_{i=1}^{p}|b_{i}||b_{j}- b_{j,p}| + \sum_{j=1}^{p}\sum_{i=1}^{p}|b_{j,p}||b_i - b_{i,p}| \bigg)\sum_{k\in\mathbb{Z}}|R_k|\end{aligned}\\
&\!
\begin{aligned}[t]= \left( \sum_{i=1}^{p}|b_{i}|+\sum_{i=1}^{p}|b_{i,p}|\right)\left(\sum_{i=1}^{p}|b_{i}- b_{i,p}|\right)\sum_{k\in\mathbb{Z}}|R_k| \end{aligned}
\end{split}
\end{equation}
As the covariance sequence has been assumed to be in $\ell_1$, the sequence $\{b_k\}$ is also in $\ell_1$ (\cite{Bhansali}) and by Lemma (\ref{Baxter}), as $p\to\infty$ the second term goes to $0$. Finally, by the same lemma, 

\be \nonumber  \lim_{p\to\infty}\sum_{i=1}^{p}|b_{i,p}| \le  \lim_{p\to\infty} \sum_{i=1}^{p}|b_{i}|\ee 

Therefore,
\be\label{AR L2 1} \lim_{p\to\infty} \sum_{k\in\mathbb{Z}}|R_{k,p} - \overline{R}_{k,p}| = 0 \ee

Now note that
\ber R_{k,p} &=& \mean[(X_n - \nu_n- \sum_{i=p+1}^\infty b_iX_{n-i})(X_{n+k} \notag\\
  &-& \nu_{n+k}- \sum_{j=p+1}^\infty b_jX_{n+k-j})]\nonumber\\
&=& \overline{R}_k -\sum_{j=p+1}^\infty b_jR_{k-j}-\sum_{i=p+1}^\infty b_iR_{k+i} \notag\\
  &+&\sum_{j=p+1}^\infty b_j\mean[X_{n+k-j}\nu_n] \nonumber\\
&+&  \sum_{i=p+1}^\infty b_i\mean[X_{n-i}\nu_{n+k}] \nonumber\\ &+&\mean\left[\left(\sum_{i=p+1}^\infty b_iX_{n-i}\right)\left(\sum_{j=p+1}^\infty b_jX_{n+k-j}\right)\right]\nonumber\\
&=& \overline{R}_k -\sum_{j=p+1}^\infty b_jR_{k-j}-\sum_{i=p+1}^\infty b_iR_{k+i} \notag\\
  &+&\sum_{j=p+1}^\infty b_ja_{|k|-j} + \left(\sum_{i=p+1}^\infty b_i^2\right) R_k \nonumber\\
&+&  \sum_{t=1}^\infty\sum_{i=p+1}^\infty b_ib_{i+t} (R_{k+t}+R_{k-t})\nonumber\\
\eer
Therefore,
\begin{equation}\begin{split}
&\!
\begin{aligned}[t] && |R_{k,p}- \overline{R}_k|  \end{aligned}\\
&\!
\begin{aligned}[t] \le  \sum_{j=p+1}^\infty |b_j||R_{k-j}|+\sum_{i=p+1}^\infty |b_i||R_{k+i}|\end{aligned}\\
&\!
\begin{aligned}[t]+  \sum_{j=p+1}^\infty |b_j||a_{|k|-j}|+ \left(\sum_{i=p+1}^\infty |b_i|^2\right) |R_k| \end{aligned}\\
&\!
\begin{aligned}[t]+ \sum_{t=1}^\infty\sum_{i=p+1}^\infty |b_i||b_{i+t}|(|R_{k+t}|+|R_{k-t}|)\nonumber\\  \end{aligned}\\
\end{split}
\end{equation}
Summing over all $k\in \mathbb{Z}$,

\begin{equation}\begin{split}
&\!
\begin{aligned}[t] && \sum_{k\in\mathbb{Z}}|R_{k,p}- \overline{R}_k| \end{aligned}\\
&\!
\begin{aligned}[t]\le  \left(\sum_{j=p+1}^\infty |b_j|\right)\sum_{k\in\mathbb{Z}}|R_{k-j}|+ \left(\sum_{i=p+1}^\infty |b_i|\right)\sum_{k\in\mathbb{Z}}|R_{k+i}| \end{aligned}\\
&\!
\begin{aligned}[t]+ \left(\sum_{j=p+1}^\infty |b_j|\right)\sum_{k\in\mathbb{Z}}|a_{|k|-j}|+ \left(\sum_{i=p+1}^\infty |b_i|^2\right) \sum_{k\in\mathbb{Z}}|R_k| \end{aligned}\\
&\!
\begin{aligned}[t]+ \left(\sum_{t=1}^\infty\sum_{i=p+1}^\infty |b_i||b_{i+t}|\right) \bigg(\sum_{k\in\mathbb{Z}}|R_{k+t}|+ \sum_{k\in\mathbb{Z}}|R_{k-t}|\bigg)
\end{aligned}\\
&\!
\begin{aligned}[t]= \left(\sum_{j=p+1}^\infty |b_j|\right)\left(2\sum_{k\in\mathbb{Z}}|R_k| + \sum_{k=0}^\infty |a_k|\right) \end{aligned}\\
&\!
\begin{aligned}[t]+ \left(\sum_{i=p+1}^\infty |b_i|\right)^2 \sum_{k\in\mathbb{Z}}|R_k|\end{aligned}\\
\end{split}
\end{equation}
As $p\to\infty$, each term on the right hand side of the above inequality goes to zero, because the covariance sequence and the sequences $\{b_k\}$ and $\{a_k\}$ are in $\ell_1$. Therefore,

\be\label{AR L2 2} \lim_{p\to\infty} \sum_{k\in\mathbb{Z}}|R_{k,p}- \overline{R}_k| = 0 \ee

Combining the results of (\ref{AR L2 1}) and (\ref{AR L2 2}) we obtain
\be\label{last} \lim_{p\to\infty} \sum_{k\in\mathbb{Z}}|\overline{R}_{k,p}- \overline{R}_k| = 0 \ee
Finally, 
\begin{equation}\begin{split}
&\!
\begin{aligned}[t]
&& && \lim_{p\to\infty}||S_{\overline{X}_p}(\lambda)-S_{\overline{X}}(\lambda)||\end{aligned}\\
&\!
\begin{aligned}[t]
 &&= \lim_{p\to\infty}\left|\int_{-\frac{1}{2}}^{\frac{1}{2}}\left| \sum_{k\in\mathbb{Z}}(\overline{R}_{k,p}-\overline{R}_k)e^{-2\pi i \lambda k}\right|^2d\lambda\right|^{\frac{1}{2}}\end{aligned}\\
&\!
\begin{aligned}[t]
 &&\le \lim_{p\to\infty}\left|\int_{-\frac{1}{2}}^{\frac{1}{2}}\sum_{k\in\mathbb{Z}}|\overline{R}_{k,p}-\overline{R}_k|^2d\lambda\right|^{\frac{1}{2}}\end{aligned}\\
&\!
\begin{aligned}[t]
 &&\le 
\lim_{p\to\infty}\left(\sum_{k\in\mathbb{Z}}|\overline{R}_{k,p}-\overline{R}_k|\right)\end{aligned}\\
&\!
\begin{aligned}[t]
 &&= 0\end{aligned}\\
\end{split}
\end{equation}
where (\ref{last}) is used for the last equality. This completes the proof. 
\end{proof}

\begin{rem}
Combining the above result with the fact that $\nu_{n,p}$ converges to $\nu_n$ in quadratic mean it readily follows that the spectral density of $\{\overline{X}_{n,p}+\nu_{n,p}\}$ converges to the spectral density of $\{X_n\}$ in $L_2$.
\end{rem}
\section{Conclusion}
 In this paper it has been shown that the spectral density of a finite autoregressive approximation of a WSS process converges at the origin when the spectral density is strictly positive. Thus, any unbiased spectral estimator derived from a finite autoregressive approximation will converge to the spectrum of the original process at the origin. This would enable easy approximation of the TAVC, which is important in the context of steady-state simulation.  
 
  Furthermore, it has been shown that the spectral density of both the moving average and the autoregressive type approximations converge in $L_2$ when the sequence of the Wold expansion coefficients is in $\ell_1$. For a zero mean wide sense stationary processes having an infinite order moving average representation, the condition $\sum_{0\le k <\infty}|a_k|< \infty$ is met when the covariance $R_k$ tends to zero at an exponential rate as $k\to\infty$ asymptotically; i.e., there exist constants $C\in\mathbb{R}$, $\alpha\in(0,1)$ such that $|R_k|\sim C\alpha^{|k|}$ (\cite{Bhansali01}). A more trivial example is that of a process which has a finite order moving average representation.  In such cases, the spectral density of the original process can be approximated over $\lambda \in \left[ -\frac{1}{2}, \frac{1}{2}\right]$ from that of the finite order estimate.
\section*{Acknowledgment}

The research of SDG was supported in part by the Natural Sciences and Engineering Council of Canada (NSERC) through the Strategic Grant program.





\bibliographystyle{plain}

\bibliography{wold2}

%



%
\end{document}